\documentclass{article}

\usepackage{amsmath}
\usepackage{amssymb}
\usepackage{amsthm}

\newtheorem{theorem}{Theorem}[section]
\newtheorem{proposition}[theorem]{Proposition}

\newtheorem{lemma}[theorem]{Lemma}
\theoremstyle{definition}

\begin{document}

\title{A $G$-family of quandles and handlebody-knots}
\author{Atsushi Ishii, Masahide Iwakiri, Yeonhee Jang and Kanako Oshiro}
\date{}

\maketitle

\begin{abstract}
We introduce the notion of a $G$-family of quandles which
is an algebraic system whose axioms are motivated by handlebody-knot
theory, and use it to construct invariants for handlebody-knots.
Our invariant can detect the chiralities of some handlebody-knots
including unknown ones.
\end{abstract}

\section{Introduction}
\label{sec:introduction}

A quandle \cite{Joyce82,Matveev82} is an algebraic system whose axioms
are motivated by knot theory.
Carter, Jelsovsky, Kamada, Langford, and Saito
\cite{CarterJelsovskyKamadaLangfordSaito03} defined the quandle homology
theory and quandle cocycle invariants for links and surface-links.
The quandle chain complex in \cite{CarterJelsovskyKamadaLangfordSaito03}
is a subcomplex of the rack chain complex in
\cite{FennRourkeSanderson95}.
The quandle cocycle invariant extracts information from quandle
colorings by a quandle cocycle, and are used to detect the chirality of
links in \cite{CarterKamadaSaito01,RourkeSanderson}.

In this paper, we introduce the notion of a $G$-family of quandles which
is an algebraic system whose axioms are motivated by handlebody-knot
theory, and use it to construct invariants for handlebody-knots.
A \textit{handlebody-knot} is a handlebody embedded in the 3-sphere.
A handlebody-knot can be represented by its trivalent spine, and the
first author, in \cite{Ishii08}, gave a list of local moves connecting
diagrams of spatial trivalent graphs which represent equivalent
handlebody-knots.
The axioms of a $G$-family of quandles are derived from the local
moves.

A $G$-family of quandles gives us not only invariants for
handlebody-knots but also a way to handle a number of quandles at once.
We see that a $G$-family of quandles is indeed a family of quandles
associated with a group $G$.
Any quandle is contained in some $G$-family of quandles as we see in
Proposition~\ref{prop:examples_G-family}.
We introduce a homology theory for $G$-families of quandles.
A cocycle of a $G$-family of quandles gives a family of cocycles of
quandles.
Thus it is efficient to find cocycles of a $G$-family of quandles, and
indeed Nosaka \cite{Nosaka} gave some cocycles together with a method to
construct a cocycle of a $G$-family of quandles induced by a
$G$-invariant group cocycle.

A $G$-family of quandles induces a quandle which contains all quandles
forming the $G$-family of quandles as subquandles.
This quandle, which we call the associated quandle, has a suitable
structure to define colorings of a diagram of a handlebody-knot.
Putting weights on colorings with a cocycle of a $G$-family of quandles,
we define a quandle cocycle invariant for handlebody-knots.
In \cite{IshiiIwakiri12}, the first and second authors defined quandle
colorings and quandle cocycle invariants for handlebody-links by
introducing the notion of an $A$-flow for an abelian group $A$.
Quandle cocycle invariants we define in this paper are nonabelian
versions of the invariants.
A usual knot can be regarded as a genus one handlebody-knot by taking
its regular neighborhood, and some knot invariants have been modified
and generalized to construct invariants for handlebody-knots.
In \cite{JangOshiro}, the third and fourth authors defined symmetric
quandle colorings and symmetric quandle cocycle invariants for
handlebody-links by generalizing symmetric quandle cocycle invariants of
classical knots given in \cite{Kamada07,KamadaOshiro}.

A table of genus two handlebody-knots with up to 6 crossings is given in
\cite{IshiiKishimotoMoriuchiSuzuki11}, and the handlebody-knots
$0_1,\ldots,6_{16}$ in the table were proved to be mutually distinct by
using the fundamental groups of their complements, quandle cocycle
invariants in \cite{IshiiIwakiri12} and some topological arguments in
\cite{IshiiKishimotoOzawa,LeeLee}.
Our quandle cocycle invariant can distinguish the handlebody-knots
$6_{14}$ and $6_{15}$ whose complements have isomorphic fundamental
groups, and detect the chiralities of the handlebody-knots $5_2$, $5_3$,
$6_5$, $6_9$, $6_{11}$, $6_{12}$, $6_{13}$, $6_{14}$, $6_{15}$.
In particular, the chiralities of $5_3$, $6_5$, $6_{11}$ and $6_{12}$
were not known.

This paper is organized as follows.
In Section~\ref{sec:G-family}, we give the definition of a $G$-family of
quandles together with some examples.
In Section~\ref{sec:coloring}, we describe colorings with a $G$-family
of quandles for handlebody-links.
We define the homology for a $G$-family of quandles in
Section~\ref{sec:homology} and define several invariants for
handlebody-links including quandle cocycle invariants in
Section~\ref{sec:cocycle_invariants}.
In Section~\ref{sec:applications}, we calculate quandle cocycle
invariants for handlebody-knots with up to $6$ crossings and show the
chirality for some of the handlebody-knots.
In Section~\ref{sec:generalization}, we prove that our invariants can be
regarded as a generalization of the invariants defined in
\cite{IshiiIwakiri12}.

\section{A $G$-family of quandles}
\label{sec:G-family}

A \textit{quandle}~\cite{Joyce82,Matveev82} is a non-empty set $X$ with
a binary operation $*:X\times X\to X$ satisfying the following axioms.
\begin{itemize}
\item For any $x\in X$, $x*x=x$.
\item For any $x\in X$, the map $S_x:X\to X$ defined by $S_x(y)=y*x$ is
a bijection.
\item For any $x,y,z\in X$, $(x*y)*z=(x*z)*(y*z)$.
\end{itemize}
A \textit{rack} is a non-empty set $X$ with a binary operation
$*:X\times X\to X$ satisfying the second and third axioms.
When we specify the binary operation $*$ of a quandle (resp.~rack) $X$,
we denote the quandle (resp.~rack) by the pair $(X,*)$.
An \textit{Alexander quandle} $(M,*)$ is a $\Lambda$-module $M$ with the
binary operation defined by $x*y=tx+(1-t)y$, where
$\Lambda:=\mathbb{Z}[t,t^{-1}]$.
A \textit{conjugation quandle} $(G,*)$ is a group $G$ with the binary
operation defined by $x*y=y^{-1}xy$.

Let $G$ be a group with identity element $e$.
A \textit{$G$-family of quandles} is a non-empty set $X$ with a family
of binary operations $*^g:X\times X\to X$ ($g\in G$) satisfying the
following axioms.
\begin{itemize}
\item For any $x\in X$ and any $g\in G$, $x*^gx=x$.
\item For any $x,y\in X$ and any $g,h\in G$,
\[ x*^{gh}y=(x*^gy)*^hy\text{ and }x*^ey=x. \]
\item For any $x,y,z\in X$ and any $g,h\in G$,
\[ (x*^gy)*^hz=(x*^hz)*^{h^{-1}gh}(y*^hz). \]
\end{itemize}
When we specify the family of binary operations $*^g:X\times X\to X$
($g\in G$) of a $G$-family of quandles, we denote the $G$-family of
quandles by the pair $(X,\{*^g\}_{g\in G})$.

\begin{proposition} \label{prop:associated_quandle}
Let $G$ be a group.
Let $(X,\{*^g\}_{g\in G})$ be a $G$-family of quandles.
\begin{itemize}
\item[(1)]
For each $g\in G$, the pair $(X,*^g)$ is a quandle.
\item[(2)]
We define a binary operation
$*:(X\times G)\times(X\times G)\to X\times G$ by
\[ (x,g)*(y,h)=(x*^hy,h^{-1}gh). \]
Then $(X\times G,*)$ is a quandle.
\end{itemize}
\end{proposition}

We call the quandle $(X\times G,*)$ in
Proposition~\ref{prop:associated_quandle} the
\textit{associated quandle} of $X$.
We note that the involution $f:X\times G\to X\times G$ defined by
$f((x,g))=(x,g^{-1})$ is a good involution of the associated quandle
$X\times G$, where we refer the reader to~\cite{Kamada07} for the
definition of a good involution of a quandle.
Before proving this proposition, we introduce a notion of a $Q$-family
of quandles.
Let $(Q,\triangleleft)$ be a quandle.
A \textit{$Q$-family of quandles} is a non-empty set $X$ with a family
of binary operations $*^a:X\times X\to X$ ($a\in Q$) satisfying the
following axioms.
\begin{itemize}
\item For any $x\in X$ and any $a\in Q$, $x*^ax=x$.
\item For any $x\in X$ and any $a\in Q$, the map $S_{x,a}:X\to X$
defined by $S_{x,a}(y)=y*^ax$ is a bijection.
\item For any $x,y,z\in X$ and any $a,b\in Q$,
$(x*^ay)*^bz=(x*^bz)*^{a\triangleleft b}(y*^bz)$.
\end{itemize}
Let $Q$ be a rack.
A \textit{$Q$-family of racks} is a non-empty set $X$ with a family of
binary operations $*^a:X\times X\to X$ ($a\in Q$) satisfying the second
and third axioms.

\begin{lemma} \label{lem:Q-family}
Let $(Q,\triangleleft)$ be a quandle (resp.~rack).
Let $(X,\{*^a\}_{a\in Q})$ be a $Q$-family of quandles (resp.~racks).
We define a binary operation
$*:(X\times Q)\times(X\times Q)\to X\times Q$ by
\[ (x,a)*(y,b)=(x*^by,a\triangleleft b). \]
Then $(X\times Q,*)$ is a quandle (resp.~rack).
\end{lemma}

\begin{proof}
The first axiom of a quandle follows from the equalities
\[ (x,a)*(x,a)=(x*^ax,a\triangleleft a)=(x,a). \]
For any $(x,a),(y,b)\in X\times Q$, there is a unique
$(z,c)\in X\times Q$ such that $x=z*^by$ and $a=c\triangleleft b$.
By the equalities $(x,a)=(z*^by,c\triangleleft b)=(z,c)*(y,b)$, we have
the second axiom of a quandle.
The third axiom of a quandle follows from
\begin{align*}
((x,a)*(y,b))*(z,c)
&=((x*^by)*^cz,(a\triangleleft b)\triangleleft c) \\
&=((x*^cz)*^{b\triangleleft c}(y*^cz),(a\triangleleft c)\triangleleft(b\triangleleft c)) \\
&=((x,a)*(z,c))*((y,b)*(z,c)).
\end{align*}
\end{proof}

Conversely, we can prove the following.
Let $\triangleleft$ be a binary operation on a non-empty set $Q$.
Let $*^a$ be a binary operation on a non-empty set $X$ for $a\in Q$.
We define a binary operation
$*:(X\times Q)\times(X\times Q)\to X\times Q$ by
\[ (x,a)*(y,b)=(x*^by,a\triangleleft b). \]
If $(X\times Q,*)$ is a quandle (resp.~rack), then $(Q,\triangleleft)$
is a quandle (resp.~rack) and $(X,\{*^a\}_{a\in Q})$ is a $Q$-family of
quandles (resp.~racks).

\begin{proof}[Proof of Proposition~\ref{prop:associated_quandle}]
\begin{itemize}
\item[(1)]
The first and third axioms of a quandle are easily checked.
The second axiom of a quandle follows from the equalities
\[ (x*^gy)*^{g^{-1}}y=(x*^{g^{-1}}y)*^gy=x. \]
Then $(X,*^g)$ is a quandle.
\item[(2)]
Let $(G,\triangleleft)$ be the conjugation quandle.
By Lemma~\ref{lem:Q-family}, $(X\times G,*)$ is a quandle.
\end{itemize}
\end{proof}

The following proposition gives us many examples for a $G$-family of
quandles.

\begin{proposition} \label{prop:examples_G-family}
\begin{itemize}
\item[(1)]
Let $(X,*)$ be a quandle.
Let $S_x:X\to X$ be the bijection defined by $S_x(y)=y*x$.
Let $m$ be a positive integer such that $S_x^m=\mathrm{id}_X$ for any
$x\in X$ if such an integer exists.
We define the binary operation $*^i:X\times X\to X$ by $x*^iy=S_y^i(x)$.
Then $X$ is a $\mathbb{Z}$-family of quandles and a
$\mathbb{Z}_m$-family of quandles, where
$\mathbb{Z}_m=\mathbb{Z}/m\mathbb{Z}$.
\item[(2)]
Let $R$ be a ring, and $G$ a group with identity element $e$.
Let $X$ be a right $R[G]$-module, where $R[G]$ is the group ring of $G$
over $R$.
We define the binary operation $*^g:X\times X\to X$ by
$x*^gy=xg+y(e-g)$.
Then $X$ is a $G$-family of quandles.
\end{itemize}
\end{proposition}

\begin{proof}
\begin{itemize}
\item[(1)]
We verify the axioms of a $G$-family of quandles.
\begin{align*}
x*^0y&=S_y^0(x)=\mathrm{id}_X(x)=x, \\
x*^ix&=S_x^i(x)=x, \\
(x*^iy)*^jy&=S_y^j(S_y^i(x))=S_y^{i+j}(x)=x*^{i+j}y.
\end{align*}
For the last axiom of a $G$-family of quandles, we can prove
\[ (x*^jz)*^i(y*^jz)=(x*^iy)*^jz \]
by induction.
\item[(2)]
We verify the axioms of a $G$-family of quandles.
\begin{align*}
x*^ey&=xe+y(e-e)=x, \\
x*^gx&=xg+x(e-g)=x, \\
(x*^gy)*^hy
&=(xg+y-yg)h+y-yh
=x*^{gh}y,
\end{align*}
\begin{align*}
&(x*^hz)*^{h^{-1}gh}(y*^hz) \\
&=(xh+z-zh)h^{-1}gh+(yh+z-zh)-(yh+z-zh)h^{-1}gh \\
&=(xg+y-yg)h+z-zh \\
&=(x*^gy)*^hz.
\end{align*}
\end{itemize}
\end{proof}

\section{Handlebody-links and $X$-colorings}
\label{sec:coloring}

A \textit{handlebody-link} is a disjoint union of handlebodies embedded
in the $3$-sphere $S^3$.
Two handlebody-links are \textit{equivalent} if there is an
orientation-preserving self-homeomorphism of $S^3$ which sends one to
the other.
A \textit{spatial graph} is a finite graph embedded in $S^3$.
Two spatial graphs are \textit{equivalent} if there is an
orientation-preserving self-homeomorphism of $S^3$ which sends one to
the other.
When a handlebody-link $H$ is a regular neighborhood of a spatial graph
$K$, we say that $K$ \textit{represents} $H$, or $H$
\textit{is represented by} $K$.
In this paper, a trivalent graph may contain circle components.
Then any handlebody-link can be represented by some spatial trivalent
graph.
A \textit{diagram} of a handlebody-link is a diagram of a spatial
trivalent graph which represents the handlebody-link.

An \textit{IH-move} is a local spatial move on spatial trivalent graphs
as described in Figure~\ref{fig:IH-move}, where the replacement is
applied in a $3$-ball embedded in $S^3$.
Then we have the following theorem.

\begin{figure}
\mbox{}\hfill
\begin{picture}(110,40)
\put(20,10){\line(2,-1){20}}
\put(20,10){\line(-2,-1){20}}
\put(20,10){\line(0,1){20}}
\put(20,30){\line(2,1){20}}
\put(20,30){\line(-2,1){20}}
\put(55,20){\makebox(0,0){$\leftrightarrow$}}
\put(80,20){\line(-1,2){10}}
\put(80,20){\line(-1,-2){10}}
\put(80,20){\line(1,0){20}}
\put(100,20){\line(1,2){10}}
\put(100,20){\line(1,-2){10}}
\end{picture}
\hfill\mbox{}
\caption{}
\label{fig:IH-move}
\end{figure}

\begin{theorem}[\cite{Ishii08}] \label{thm:Rmove}
For spatial trivalent graphs $K_1$ and $K_2$, the following are
equivalent.
\begin{itemize}
\item $K_1$ and $K_2$ represent an equivalent handlebody-link.
\item $K_1$ and $K_2$ are related by a finite sequence of IH-moves.
\item Diagrams of $K_1$ and $K_2$ are related by a finite sequence of
the moves depicted in Figure~\ref{fig:Rmove}.
\end{itemize}
\end{theorem}

\begin{figure}
\mbox{}\hfill
\begin{picture}(100,40)
 \qbezier(0,0)(0,12)(5,20)
 \qbezier(5,20)(10,29)(14,29)
 \qbezier(14,11)(10,11)(7,16)
 \qbezier(3,24)(0,30)(0,40)
 \qbezier(14,11)(20,11)(20,20)
 \qbezier(14,29)(20,29)(20,20)
 \put(35,20){\makebox(0,0){$\leftrightarrow$}}
 \qbezier(50,0)(50,20)(50,40)
 \put(65,20){\makebox(0,0){$\leftrightarrow$}}
 \qbezier(80,0)(80,10)(83,16)
 \qbezier(87,24)(90,29)(94,29)
 \qbezier(94,11)(90,11)(85,20)
 \qbezier(85,20)(80,28)(80,40)
 \qbezier(94,11)(100,11)(100,20)
 \qbezier(94,29)(100,29)(100,20)
\end{picture}
\hfill\hfill
\begin{picture}(70,40)
 \qbezier(0,0)(0,6)(10,10)
 \qbezier(10,10)(20,14)(20,20)
 \qbezier(20,0)(20,5)(14,8)
 \qbezier(6,12)(0,15)(0,20)
 \qbezier(0,20)(0,25)(6,28)
 \qbezier(14,32)(20,35)(20,40)
 \qbezier(20,20)(20,26)(10,30)
 \qbezier(10,30)(0,34)(0,40)
 \put(35,20){\makebox(0,0){$\leftrightarrow$}}
 \qbezier(50,0)(50,20)(50,40)
 \qbezier(70,0)(70,20)(70,40)
\end{picture}
\hfill\hfill
\begin{picture}(80,40)
 \qbezier(0,0)(0,4)(6.66,6.67)
 \qbezier(6.66,6.67)(13.33,9.33)(13.33,13.33)
 \qbezier(13.33,0)(13.33,3.33)(9.33,5.33)
 \qbezier(4,8)(0,10)(0,13.33)
 \qbezier(26.66,0)(26.66,6.67)(26.66,13.33)
 \qbezier(0,13.33)(0,20)(0,26.67)
 \qbezier(13.33,13.33)(13.33,17.33)(20,20)
 \qbezier(20,20)(26.66,22.67)(26.66,26.67)
 \qbezier(26.66,13.33)(26.66,16.67)(22.66,18.67)
 \qbezier(17.33,21.33)(13.33,23.33)(13.33,26.67)
 \qbezier(0,26.67)(0,30.67)(6.66,33.33)
 \qbezier(6.66,33.33)(13.33,36)(13.33,40)
 \qbezier(13.33,26.67)(13.33,30)(9.33,32)
 \qbezier(4,34.67)(0,36.67)(0,40)
 \qbezier(26.66,26.67)(26.66,33.33)(26.66,40)
 \put(40,20){\makebox(0,0){$\leftrightarrow$}}
 \qbezier(53.33,0)(53.33,6.67)(53.33,13.33)
 \qbezier(66.66,0)(66.66,4)(73.33,6.67)
 \qbezier(73.33,6.67)(80,9.33)(80,13.33)
 \qbezier(80,0)(80,3.33)(76,5.33)
 \qbezier(70.66,8)(66.66,10)(66.66,13.33)
 \qbezier(53.33,13.33)(53.33,17.33)(60,20)
 \qbezier(60,20)(66.66,22.67)(66.66,26.67)
 \qbezier(66.66,13.33)(66.66,16.67)(62.66,18.67)
 \qbezier(57.33,21.33)(53.33,23.33)(53.33,26.67)
 \qbezier(80,13.33)(80,20)(80,26.67)
 \qbezier(53.33,26.67)(53.33,33.33)(53.33,40)
 \qbezier(66.66,26.67)(66.66,30.67)(73.33,33.33)
 \qbezier(73.33,33.33)(80,36)(80,40)
 \qbezier(80,26.67)(80,30)(76,32)
 \qbezier(70.66,34.67)(66.66,36.67)(66.66,40)
\end{picture}
\hfill\mbox{}\vspace{15pt}\\
\begin{picture}(115,40)
 \qbezier(0,0)(0,12)(5,20)
 \qbezier(5,20)(10,29)(12,29)
 \qbezier(12,11)(10,11)(7,16)
 \qbezier(3,24)(0,30)(0,40)
 \qbezier(12,11)(15,11)(15,20)
 \qbezier(12,29)(15,29)(15,20)
 \qbezier(15,20)(20,20)(25,20)
 \put(37.5,20){\makebox(0,0){$\leftrightarrow$}}
 \qbezier(50,0)(52.5,10)(55,20)
 \qbezier(50,40)(52.5,30)(55,20)
 \qbezier(55,20)(60,20)(65,20)
 \put(77.5,20){\makebox(0,0){$\leftrightarrow$}}
 \qbezier(90,40)(90,28)(95,20)
 \qbezier(95,20)(100,11)(102,11)
 \qbezier(102,29)(100,29)(97,24)
 \qbezier(93,16)(90,10)(90,0)
 \qbezier(102,29)(105,29)(105,20)
 \qbezier(102,11)(105,11)(105,20)
 \qbezier(105,20)(110,20)(115,20)
\end{picture}
\hfill\hfill
\begin{picture}(110,40)
 \qbezier(10,0)(10,20)(10,40)
 \qbezier(0,20)(0,27)(8,33)\qbezier(12,35)(16,37.5)(20,40)
 \qbezier(0,20)(4,20)(8,20)\qbezier(12,20)(16,20)(20,20)
 \qbezier(0,20)(0,13)(8,7)\qbezier(12,5)(16,2.5)(20,0)
 \put(32.5,20){\makebox(0,0){$\leftrightarrow$}}
 \qbezier(45,0)(45,20)(45,40)
 \qbezier(55,20)(55,30)(65,40)
 \qbezier(55,20)(55,20)(65,20)
 \qbezier(55,20)(55,10)(65,0)
 \put(77.5,20){\makebox(0,0){$\leftrightarrow$}}
 \qbezier(100,0)(100,2)(100,4)\qbezier(100,8)(100,15)(100,18)
 \qbezier(100,22)(100,29)(100,32)\qbezier(100,36)(100,38)(100,40)
 \qbezier(90,20)(90,30)(110,40)
 \qbezier(90,20)(100,20)(110,20)
 \qbezier(90,20)(90,10)(110,0)
\end{picture}
\hfill\hfill
\begin{picture}(60,40)
 \qbezier(0,40)(5,35)(10,30)
 \qbezier(20,40)(15,35)(10,30)
 \qbezier(10,30)(10,20)(10,10)
 \qbezier(10,10)(5,5)(0,0)
 \qbezier(10,10)(15,5)(20,0)
 \put(30,20){\makebox(0,0){$\leftrightarrow$}}
 \qbezier(40,40)(42.5,30)(45,20)
 \qbezier(40,0)(42.5,10)(45,20)
 \qbezier(45,20)(50,20)(55,20)
 \qbezier(55,20)(57.5,30)(60,40)
 \qbezier(55,20)(57.5,10)(60,0)
\end{picture}\vspace*{5pt}
\caption{}
\label{fig:Rmove}
\end{figure}

Let $D$ be a diagram of a handlebody-link $H$.
We set an orientation for each edge in $D$.
Then $D$ is a diagram of an oriented spatial trivalent graph $K$.
We may represent an orientation of an edge by a normal orientation,
which is obtained by rotating a usual orientation counterclockwise by
$\pi/2$ on the diagram.
We denote by $\mathcal{A}(D)$ the set of arcs of $D$, where an arc is a
piece of a curve each of whose endpoints is an undercrossing or a
vertex.
For an arc $\alpha$ incident to a vertex $\omega$, we define
$\epsilon(\alpha;\omega)\in\{1,-1\}$ by
\[ \epsilon(\alpha;\omega)=\begin{cases}
1 & \text{if the orientation of $\alpha$ points to $\omega$,} \\
-1 & \text{otherwise.}
\end{cases} \]
Let $X$ be a $G$-family of quandles, and $Q$ the associated quandle of
$X$.
Let $p_X$ (resp.~$p_G$) be the projection from $Q$ to $X$ (resp.~$G$).
An \textit{$X$-coloring} of $D$ is a map $C:\mathcal{A}(D)\to Q$
satisfying the following conditions at each crossing $\chi$ and each
vertex $\omega$ of $D$ (see Figure~\ref{fig:X-coloring}).
\begin{itemize}
\item
Let $\chi_1,\chi_2$ and $\chi_3$ be respectively the under-arcs and the
over-arc at a crossing $\chi$ such that the normal orientation of
$\chi_3$ points from $\chi_1$ to $\chi_2$.
Then
\[ C(\chi_2)=C(\chi_1)*C(\chi_3). \]
\item
Let $\omega_1,\omega_2,\omega_3$ be the arcs incident to a vertex
$\omega$ arranged clockwise around $\omega$.
Then
\begin{align*}
&(p_X\circ C)(\omega_1)=(p_X\circ C)(\omega_2)=(p_X\circ C)(\omega_3), \\
&(p_G\circ C)(\omega_1)^{\epsilon(\omega_1;\omega)}
(p_G\circ C)(\omega_2)^{\epsilon(\omega_2;\omega)}
(p_G\circ C)(\omega_3)^{\epsilon(\omega_3;\omega)}=e.
\end{align*}
\end{itemize}
We denote by $\operatorname{Col}_X(D)$ the set of $X$-colorings of $D$.
We call $C(\alpha)$ the \textit{color} of $\alpha$.
For two diagrams $D$ and $E$ which locally differ, we denote by
$\mathcal{A}(D,E)$ the set of arcs that $D$ and $E$ share.

\begin{figure}
\mbox{}\hfill
\begin{picture}(110,90)
\put(15,45){\line(1,0){25}}
\put(50,45){\line(1,0){25}}
\put(45,15){\line(0,1){60}}
\put(46,25){\makebox(0,0){$\rightarrow$}}
\put(45,5){\makebox(0,0){$q_3$}}
\put(5,45){\makebox(0,0){$q_1$}}
\put(95,45){\makebox(0,0){$q_1*q_3$}}
\end{picture}
\hfill
\begin{picture}(86,90)
\put(15,15){\line(1,1){28}}
\put(71,15){\line(-1,1){28}}
\put(43,75){\line(0,-1){32}}
\put(26,25){\makebox(0,0){$\searrow$}}
\put(62,25){\makebox(0,0){$\nearrow$}}
\put(44,65){\makebox(0,0){$\rightarrow$}}
\put(15,5){\makebox(0,0){$(x,g)$}}
\put(71,5){\makebox(0,0){$(x,h)$}}
\put(45,85){\makebox(0,0){$(x,gh)$}}
\end{picture}
\hfill\mbox{}\vspace{5mm}\\
\mbox{}\hfill
\begin{picture}(86,90)
\put(15,15){\line(1,1){28}}
\put(71,15){\line(-1,1){28}}
\put(43,75){\line(0,-1){32}}
\put(26,25){\makebox(0,0){$\searrow$}}
\put(62,25){\makebox(0,0){$\nearrow$}}
\put(42,65){\makebox(0,0){$\leftarrow$}}
\put(15,5){\makebox(0,0){$(x,g)$}}
\put(71,5){\makebox(0,0){$(x,h)$}}
\put(45,85){\makebox(0,0){$(x,h^{-1}g^{-1})$}}
\end{picture}
\hfill
\begin{picture}(86,90)
\put(15,15){\line(1,1){28}}
\put(71,15){\line(-1,1){28}}
\put(43,75){\line(0,-1){32}}
\put(24,25){\makebox(0,0){$\nwarrow$}}
\put(60,25){\makebox(0,0){$\swarrow$}}
\put(42,65){\makebox(0,0){$\leftarrow$}}
\put(15,5){\makebox(0,0){$(x,g)$}}
\put(71,5){\makebox(0,0){$(x,h)$}}
\put(45,85){\makebox(0,0){$(x,hg)$}}
\end{picture}
\hfill
\begin{picture}(86,90)
\put(15,15){\line(1,1){28}}
\put(71,15){\line(-1,1){28}}
\put(43,75){\line(0,-1){32}}
\put(24,25){\makebox(0,0){$\nwarrow$}}
\put(60,25){\makebox(0,0){$\swarrow$}}
\put(44,65){\makebox(0,0){$\rightarrow$}}
\put(15,5){\makebox(0,0){$(x,g)$}}
\put(71,5){\makebox(0,0){$(x,h)$}}
\put(45,85){\makebox(0,0){$(x,g^{-1}h^{-1})$}}
\end{picture}
\hfill\mbox{}
\caption{}
\label{fig:X-coloring}
\end{figure}

\begin{lemma} \label{lem:X-coloring}
Let $X$ be a $G$-family of quandles.
Let $D$ be a diagram of an oriented spatial trivalent graph.
Let $E$ be a diagram obtained by applying one of the R1--R6 moves to the
diagram $D$ once, where we choose orientations for $E$ which agree with
those for $D$ on $\mathcal{A}(D,E)$.
For $C\in\operatorname{Col}_X(D)$, there is a unique $X$-coloring
$C_{D,E}\in\operatorname{Col}_X(E)$ such that
$C|_{\mathcal{A}(D,E)}=C_{D,E}|_{\mathcal{A}(D,E)}$.
\end{lemma}

\begin{proof}
The color of an edge in $\mathcal{A}(E)-\mathcal{A}(D,E)$ is uniquely
determined by the colors of edges in $\mathcal{A}(D,E)$, since we have
\[ a*^ga=a \]
for the R1, R4 moves, and
\[ (a*^gb)*^{g^{-1}}b=a*^eb=a \]
for the R2 move, and
\[ (a*^gb)*^hc =(a*^hc)*^{h^{-1}gh}(b*^hc) \]
for the R3 move, and
\[ ((b*^ga)*^ha)*^{(gh)^{-1}}a=a*^eb=b \]
for the R5 move, and only the coloring condition for the R6-move.
\end{proof}

Let $X$ be a $G$-family of quandles.
An \textit{$X$-set} is a non-empty set $Y$ with a family of maps
$*^g:Y\times X\to Y$ satisfying the following axioms, where we note that
we use the same symbol $*^g$ as the binary operation of the $G$-family
of quandles.
\begin{itemize}
\item
For any $y\in Y$, $x\in X$, and any $g,h\in G$,
\[ y*^{gh}x=(y*^gx)*^hx\text{ and }y*^ex=y. \]
\item
For any $y\in Y$, $x_1,x_2\in X$, and any $g,h\in G$,
\[ (y*^gx_1)*^hx_2=(y*^hx_2)*^{h^{-1}gh}(x_1*^hx_2). \]
\end{itemize}
Any $G$-family of quandles $(X,\{*^g\}_{g\in G})$ itself is an $X$-set
with its binary operations.
Any singleton set $\{y\}$ is also an $X$-set with the maps $*^g$ defined
by $y*^gx=y$ for $x\in X$ and $g\in G$, which is a trivial $X$-set.

Let $D$ be a diagram of an oriented spatial trivalent graph.
We denote by $\mathcal{R}(D)$ the set of complementary regions of $D$.
Let $X$ be a $G$-family of quandles, and $Y$ an $X$-set.
Let $Q$ be the associated quandle of $X$.
An \textit{$X_Y$-coloring} of $D$ is a map
$C:\mathcal{A}(D)\cup\mathcal{R}(D)\to Q\cup Y$ satisfying the following
conditions.
\begin{itemize}
\item
$C(\mathcal{A}(D))\subset Q$, $C(\mathcal{R}(D))\subset Y$.
\item
The restriction $C|_{\mathcal{A}(D)}$ of $C$ on $\mathcal{A}(D)$ is an
$X$-coloring of $D$.
\item
For any arc $\alpha\in\mathcal{A}(D)$, we have
\[ C(\alpha_1)*C(\alpha)=C(\alpha_2), \]
where $\alpha_1,\alpha_2$ are the regions facing the arc $\alpha$ so
that the normal orientation of $\alpha$ points from $\alpha_1$ to
$\alpha_2$ (see Figure~\ref{fig:XY-coloring}).
\end{itemize}
We denote by $\operatorname{Col}_X(D)_Y$ the set of $X_Y$-colorings of
$D$.

\begin{figure}
\mbox{}\hfill
\begin{picture}(90,90)
\put(10,35){\framebox(20,20){$y_1$}}
\put(60,35){\framebox(35,20){$y_1*q$}}
\put(45,15){\line(0,1){60}}
\put(46,25){\makebox(0,0){$\rightarrow$}}
\put(45,5){\makebox(0,0){$q$}}
\end{picture}
\hfill\mbox{}
\caption{}
\label{fig:XY-coloring}
\end{figure}

For two diagrams $D$ and $E$ which locally differ, we denote by
$\mathcal{R}(D,E)$ the set of regions that $D$ and $E$ share.
Since colors of regions are uniquely determined by those of arcs and one
region, Lemma~\ref{lem:X-coloring} implies the following lemma.

\begin{lemma} \label{lem:XY-coloring}
Let $X$ be a $G$-family of quandles, $Y$ an $X$-set.
Let $D$ be a diagram of an oriented spatial trivalent graph.
Let $E$ be a diagram obtained by applying one of the R1--R6 moves to the
diagram $D$ once, where we choose orientations for $E$ which agree with
those for $D$ on $\mathcal{A}(D,E)$.
For $C\in\operatorname{Col}_X(D)_Y$, there is a unique $X_Y$-coloring
$C_{D,E}\in\operatorname{Col}_X(E)_Y$ such that
$C|_{\mathcal{A}(D,E)}=C_{D,E}|_{\mathcal{A}(D,E)}$ and
$C|_{\mathcal{R}(D,E)}=C_{D,E}|_{\mathcal{R}(D,E)}$.
\end{lemma}

\section{A homology}
\label{sec:homology}

Let $X$ be a $G$-family of quandles, and $Y$ an $X$-set.
Let $(Q,*)$ be the associated quandle of $X$.
Let $B_n(X)_Y$ be the free abelian group generated by the elements of
$Y\times Q^n$ if $n\geq0$, and let $B_n(X)_Y=0$ otherwise.
We put
\[ ((y,q_1,\ldots,q_i)*q,q_{i+1},\ldots,q_n)
:=(y*q,q_1*q,\ldots,q_i*q,q_{i+1},\ldots,q_n) \]
for $y\in Y$ and $q,q_1\ldots,q_n\in Q$.
We define a boundary homomorphism
$\partial_n:B_n(X)_Y\to B_{n-1}(X)_Y$ by
\begin{align*}
\partial_n(y,q_1,\ldots,q_n)
=&\sum_{i=1}^{n}(-1)^i(y,q_1,\ldots,q_{i-1},q_{i+1},\ldots,q_n) \\
&-\sum_{i=1}^{n}(-1)^i((y,q_1,\ldots,q_{i-1})*q_i,q_{i+1},\ldots,q_n)
\end{align*}
for $n>0$, and $\partial_n=0$ otherwise.
Then $B_*(X)_Y=(B_n(X)_Y,\partial_n)$ is a chain complex
(see~\cite{CarterJelsovskyKamadaLangfordSaito03,CarterJelsovskyKamadaSaito01,
FennRourkeSanderson95,FennRourkeSanderson07}).

Let $D_n(X)_Y$ be the subgroup of $B_n(X)_Y$ generated by the elements
of
\[ \bigcup_{i=1}^{n-1}\left\{
(y,q_1,\ldots,q_{i-1},(x,g),(x,h),q_{i+2},\ldots,q_n)
\,\left|\,
\begin{array}{@{}l@{}}
y\in Y,\,x\in X,\,g,h\in G \\ q_1,\ldots,q_n\in Q
\end{array}
\right.\right\} \]
and
\[ \bigcup_{i=1}^n\left\{\left.
\begin{array}{@{}l@{}}
(y,q_1,\ldots,q_{i-1},(x,gh),q_{i+1},\ldots,q_n) \\
-(y,q_1,\ldots,q_{i-1},(x,g),q_{i+1},\ldots,q_n) \\
-((y,q_1,\ldots,q_{i-1})*(x,g),(x,h),q_{i+1},\ldots,q_n)
\end{array}
\,\right|\,
\begin{array}{@{}l@{}}
y\in Y,\,x\in X, \\ g,h\in G, \\ q_1,\ldots, q_n\in Q
\end{array}
\right\}. \]
We remark that
\[ (y,q_1,\ldots,q_{i-1},(x,e),q_{i+1},\ldots,q_n) \]
and
\begin{align*}
&(y,q_1,\ldots,q_{i-1},(x,g),q_{i+1},\ldots,q_n) \\
&+((y,q_1,\ldots,q_{i-1})*(x,g),(x,g^{-1}),q_{i+1},\ldots,q_n)
\end{align*}
belong to $D_n(X)_Y$.

\begin{lemma} \label{lem:subcomplex}
For $n\in\mathbb{Z}$, we have $\partial_n(D_n(X)_Y)\subset D_{n-1}(X)_Y$.
Thus $D_*(X)_Y=(D_n(X)_Y,\partial_n)$ is a subcomplex of $B_*(X)_Y$.
\end{lemma}

\begin{proof}
It is sufficient to show the equalities
\begin{align*}
&\partial_n(y,q_1,\ldots,q_{i-1},(x,g),(x,h),q_{i+2},\ldots,q_n)=0, \\
&\partial_n(y,q_1,\ldots,q_{i-1},(x,gh),q_{i+1},\ldots,q_n) \\
&=\partial_n(y,q_1,\ldots,q_{i-1},(x,g),q_{i+1},\ldots,q_n) \\
&~~+\partial_n((y,q_1,\ldots,q_{i-1})*(x,g),(x,h),q_{i+1},\ldots,q_n)
\end{align*}
in $B_{n-1}(X)_Y/D_{n-1}(X)_Y$.
We verify the first equality in the quotient group.
\begin{align*}
&\partial_n(y,q_1,\ldots,q_{i-1},(x,g),(x,h),q_{i+2},\ldots,q_n) \\
&=(-1)^i(y,q_1,\ldots,q_{i-1},(x,h),q_{i+2},\ldots,q_n) \\
&~~+(-1)^{i+1}(y,q_1,\ldots,q_{i-1},(x,g),q_{i+2},\ldots,q_n) \\
&~~-(-1)^i((y,q_1,\ldots,q_{i-1})*(x,g),(x,h),q_{i+2},\ldots,q_n) \\
&~~-(-1)^{i+1}((y,q_1,\ldots,q_{i-1},(x,g))*(x,h),q_{i+2},\ldots,q_n) \\
&=(-1)^i(y,q_1,\ldots,q_{i-1},(x,h),q_{i+2},\ldots,q_n) \\
&~~+(-1)^{i+1}(y,q_1,\ldots,q_{i-1},(x,gh),q_{i+2},\ldots,q_n) \\
&~~-(-1)^{i+1}((y,q_1,\ldots,q_{i-1})*(x,h),(x,h^{-1}gh),q_{i+2},\ldots,q_n) \\
&=0,
\end{align*}
where the first equality follows from
\[ ((y,q_1,\ldots,q_{i-1},(x,g),(x,h),q_{i+2},\ldots,q_{j-1})*q_j,q_{j+1},\ldots,q_n)=0. \]
We verify the second equality in the quotient group.
\begin{align*}
&\partial_n(y,q_1,\ldots,q_{i-1},(x,gh),q_{i+1},\ldots,q_n) \\
&=\sum_{j<i}(-1)^j(y,q_1,\ldots,q_{j-1},q_{j+1},\ldots,q_{i-1},(x,g),q_{i+1},\ldots,q_n) \\
&~~+\sum_{j<i}(-1)^j((y,q_1,\ldots,q_{j-1},q_{j+1},\ldots,q_{i-1})*(x,g),(x,h),q_{i+1},\ldots,q_n) \\
&~~+(-1)^i(y,q_1,\ldots,q_{i-1},q_{i+1},\ldots,q_n) \\
&~~+\sum_{j>i}(-1)^j(y,q_1,\ldots,q_{i-1},(x,g),q_{i+1},\ldots,q_{j-1},q_{j+1},\ldots,q_n) \\
&~~+\sum_{j>i}(-1)^j((y,q_1,\ldots,q_{i-1})*(x,g),(x,h),q_{i+1},\ldots,q_{j-1},q_{j+1},\ldots,q_n) \\
&~~-\sum_{j<i}(-1)^j((y,q_1,\ldots,q_{j-1})*q_j,q_{j+1},\ldots,q_{i-1},(x,g),q_{i+1},\ldots,q_n) \\
&~~-\sum_{j<i}(-1)^j(((y,q_1,\ldots,q_{j-1})*q_j,q_{j+1},\ldots,q_{i-1})*(x,g),(x,h),q_{i+1},\ldots,q_n) \\
&~~-(-1)^i((y,q_1,\ldots,q_{i-1})*(x,gh),q_{i+1},\ldots,q_n) \\
&~~-\sum_{j>i}(-1)^j((y,q_1,\ldots,q_{i-1},(x,gh),q_{i+1},\ldots,q_{j-1})*q_j,q_{j+1},\ldots,q_n) \\
&=\partial_n(y,q_1,\ldots,q_{i-1},(x,g),q_{i+1},\ldots,q_n) \\
&~~+\partial_n((y,q_1,\ldots,q_{i-1})*(x,g),(x,h),q_{i+1},\ldots,q_n),
\end{align*}
where the last equality follows from
\begin{align*}
&((y,q_1,\ldots,q_{i-1})*(x,gh),q_{i+1},\ldots,q_n) \\
&=(((y,q_1,\ldots,q_{i-1})*(x,g))*(x,h),q_{i+1},\ldots,q_n)
\end{align*}
and
\begin{align*}
&((y,q_1,\ldots,q_{i-1},(x,gh),q_{i+1},\ldots,q_{j-1})*q_j,q_{j+1},\ldots,q_n) \\
&=((y,q_1,\ldots,q_{i-1},(x,g),q_{i+1},\ldots,q_{j-1})*q_j,q_{j+1},\ldots,q_n) \\
&~~+(((y,q_1,\ldots,q_{i-1})*(x,g),(x,h),q_{i+1},\ldots,q_{j-1})*q_j,q_{j+1},\ldots,q_n).
\end{align*}
Then $\partial_n(D_n(X)_Y)\subset D_{n-1}(X)_Y$.
\end{proof}

We put $C_n(X)_Y=B_n(X)_Y/D_n(X)_Y$.
Then $C_*(X)_Y=(C_n(X)_Y,\partial_n)$ is a chain complex.
For an abelian group $A$, we define the cochain complex
$C^*(X;A)_Y=\operatorname{Hom}(C_*(X)_Y,A)$.
We denote by $H_n(X)_Y$ the $n$th homology group of $C_*(X)_Y$.

\section{Cocycle invariants}
\label{sec:cocycle_invariants}

Let $X$ be a $G$-family of quandles, and $Y$ an $X$-set.
Let $D$ be a diagram of an oriented spatial trivalent graph.
For an $X_Y$-coloring $C\in\operatorname{Col}_X(D)_Y$, we define the
weight $w(\chi;C)\in C_2(X)_Y$ at a crossing $\chi$ of $D$ as follows.
Let $\chi_1,\chi_2$ and $\chi_3$ be respectively the under-arcs and the
over-arc at a crossing $\chi$ such that the normal orientation of
$\chi_3$ points from $\chi_1$ to $\chi_2$.
Let $R_\chi$ be the region facing $\chi_1$ and $\chi_3$ such that the
normal orientations $\chi_1$ and $\chi_3$ point from $R_\chi$ to the
opposite regions with respect to $\chi_1$ and $\chi_3$, respectively.
Then we define
\[ w(\chi;C)=\epsilon(\chi)(C(R_\chi),C(\chi_1),C(\chi_3)), \]
where $\epsilon(\chi)\in\{1,-1\}$ is the sign of a crossing $\chi$.
We define a chain $W(D;C)\in C_2(X)_Y$ by
\[ W(D;C)=\sum_\chi w(\chi;C), \]
where $\chi$ runs over all crossings of $D$.

\begin{lemma} \label{lem:W(D;C)_2-cycle}
The chain $W(D;C)$ is a $2$-cycle of $C_*(X)_Y$.
Further, for cohomologous $2$-cocycles $\theta,\theta'$ of $C^*(X;A)_Y$,
we have $\theta(W(D;C))=\theta'(W(D;C))$.
\end{lemma}

\begin{proof}
It is sufficient to show that $W(D;C)$ is a $2$-cycle of $C_2(X)_Y$.
We denote by $\mathcal{SA}(D)$ the set of curves obtained from $D$ by
removing (small neighborhoods of) crossings and vertices.
We call a curve in $\mathcal{SA}(D)$ a \textit{semi-arc} of $D$.
We note that a semi-arc is obtained by dividing an over-arc at all
crossings.
We denote by $\mathcal{SA}(D;\chi)$ the set of semi-arcs incident to
$\chi$, where $\chi$ is a crossing or a vertex of $D$.

We define the orientation and the color of a semi-arc by those of the
arc including the semi-arc.
For a semi-arc $\alpha$, there is a unique region $R_\alpha$ facing
$\alpha$ such that the orientation of $\alpha$ points from the region
$R_\alpha$ to the opposite region with respect to $\alpha$.
For a semi-arc $\alpha$ incident to a crossing or a vertex $\chi$, we
define
\[ \epsilon(\alpha;\chi):=\begin{cases}
1 & \text{if the orientation of $\alpha$ points to $\chi$,} \\
-1 & \text{otherwise.}
\end{cases} \]

Let $\chi_1,\chi_2$ be the semi-arcs incident to a crossing $\chi$ such
that they originate from the under-arcs at $\chi$ and that the normal
orientation of the over-arc points from $\chi_1$ to $\chi_2$.
Let $\chi_3,\chi_4$ be the semi-arcs incident to a crossing $\chi$ such
that they originate from the over-arc at $\chi$ and that the normal
orientation of the under-arcs points from $\chi_3$ to $\chi_4$
(see Figure~\ref{fig:semi-arcs}).
Then we have
\begin{align*}
\partial_2(w(\chi;C))
&=-\epsilon(\chi)(C(R_{\chi_1}),C(\chi_3))+\epsilon(\chi)(C(R_{\chi_1}),C(\chi_1)) \\
&\hspace{5mm}+\epsilon(\chi)(C(R_{\chi_1})*C(\chi_1),C(\chi_3)) \\
&\hspace{5mm}-\epsilon(\chi)(C(R_{\chi_1})*C(\chi_3),C(\chi_1)*C(\chi_3)) \\
&=\sum_{\alpha\in\mathcal{SA}(D;\chi)}\epsilon(\alpha;\chi)(C(R_{\alpha}),C(\alpha)).
\end{align*}
Since
$\sum_{\alpha\in\mathcal{SA}(D;\omega)}
\epsilon(\alpha;\omega)(C(R_{\alpha}),C(\alpha))$ is an element of
$D_1(X)_Y$ for a vertex $\omega$, we have
\begin{align*}
\partial_2\left(\sum_{\chi}w(\chi;C)\right)
&=\sum_{\chi}\sum_{\alpha\in\mathcal{SA}(D;\chi)}
\epsilon(\alpha;\chi)(C(R_{\alpha}),C(\alpha)) \\
&=\sum_{\chi}\sum_{\alpha\in\mathcal{SA}(D;\chi)}
\epsilon(\alpha;\chi)(C(R_{\alpha}),C(\alpha)) \\
&\hspace{5mm}+\sum_{\omega}\sum_{\alpha\in\mathcal{SA}(D;\omega)}
\epsilon(\alpha;\omega)(C(R_{\alpha}),C(\alpha)) \\
&=\sum_{\alpha\in\mathcal{SA}(D)}
((C(R_{\alpha}),C(\alpha))-(C(R_{\alpha}),C(\alpha))) \\
&=0
\end{align*}
in $C_1(X)_Y$, where $\chi$ and $\omega$ respectively run over all
crossings and vertices of $D$.
\end{proof}

\begin{figure}
\mbox{}\hfill
\begin{picture}(90,110)
\put(15,65){\line(1,0){25}}
\put(50,65){\line(1,0){25}}
\put(45,35){\line(0,1){60}}
\put(46,45){\makebox(0,0){$\rightarrow$}}
\put(25,66){\makebox(0,0){$\uparrow$}}
\put(5,65){\makebox(0,0){$\chi_1$}}
\put(85,65){\makebox(0,0){$\chi_2$}}
\put(45,25){\makebox(0,0){$\chi_3$}}
\put(45,105){\makebox(0,0){$\chi_4$}}
\put(45,5){\makebox(0,0){$\epsilon(\chi)=1$}}
\end{picture}
\hfill
\begin{picture}(90,110)
\put(15,65){\line(1,0){25}}
\put(50,65){\line(1,0){25}}
\put(45,35){\line(0,1){60}}
\put(46,45){\makebox(0,0){$\rightarrow$}}
\put(25,64){\makebox(0,0){$\downarrow$}}
\put(5,65){\makebox(0,0){$\chi_1$}}
\put(85,65){\makebox(0,0){$\chi_2$}}
\put(45,25){\makebox(0,0){$\chi_4$}}
\put(45,105){\makebox(0,0){$\chi_3$}}
\put(45,5){\makebox(0,0){$\epsilon(\chi)=-1$}}
\end{picture}
\hfill\mbox{}
\caption{}
\label{fig:semi-arcs}
\end{figure}

We recall that, for $C\in\operatorname{Col}_X(D)_Y$, there is a unique
$X_Y$-coloring $C_{D,E}\in\operatorname{Col}_X(E)$ such that
$C|_{\mathcal{A}(D,E)}=C_{D,E}|_{\mathcal{A}(D,E)}$ and
$C|_{\mathcal{R}(D,E)}=C_{D,E}|_{\mathcal{A}(R,E)}$ by
Lemma~\ref{lem:XY-coloring}.

\begin{lemma} \label{lem:W(D;C)_Rmove}
Let $D$ be a diagram of an oriented spatial trivalent graph.
Let $E$ be a diagram obtained by applying one of the R1--R6 moves to the
diagram $D$ once, where we choose orientations for $E$ which agree with
those for $D$ on $\mathcal{A}(D,E)$.
For $C\in\operatorname{Col}_X(D)_Y$ and
$C_{D,E}\in\operatorname{Col}_X(E)_Y$ such that
$C|_{\mathcal{A}(D,E)}=C_{D,E}|_{\mathcal{A}(D,E)}$ and
$C|_{\mathcal{R}(D,E)}=C_{D,E}|_{\mathcal{R}(D,E)}$, we have
$[W(D;C)]=[W(E;C_{D,E})]\in H_2(X)_Y$.
\end{lemma}

\begin{proof}
We have the invariance under the R1, R4 and R5 moves, since the
difference between $[W(D;C)]$ and $[W(E;C_{D,E})]$ is an element of
$D_2(X)_Y$.
The invariance under the R2 move follows from the signs of the crossings
which appear in the move.
We have the invariance under the R3 move, since the difference between
$[W(D;C)]$ and $[W(E;C_{D,E})]$ is an image of $\partial_3$.
We have the invariance under the R6 move, since no crossings appear in
the move.
\end{proof}

We denote by $G_H$ (resp.~$G_K$) the fundamental group of the exterior
of a handlebody-link $H$ (resp.~a spatial graph $K$).
When $H$ is represented by $K$, the groups $G_H$ and $G_K$ are
identical.
Let $D$ be a diagram of an oriented spatial trivalent graph $K$.
By the definition of an $X_Y$-coloring $C$ of $D$, the map
$p_G\circ C|_{\mathcal{A}(D)}$ represents a homomorphism from $G_K$ to
$G$, which we denote by $\rho_C\in\operatorname{Hom}(G_K,G)$.
For $\rho\in\operatorname{Hom}(G_K,G)$, we define
\[ \operatorname{Col}_X(D;\rho)_Y
=\{C\in\operatorname{Col}_X(D)_Y\,|\,\rho_C=\rho \}. \]
For a $2$-cocycle $\theta$ of $C^*(X;A)_Y$, we define
\begin{align*}
\mathcal{H}(D)&:=\{[W(D;C)]\in H_2(X)_Y
\,|\,C\in\operatorname{Col}_X(D)_Y\}, \\
\Phi_\theta(D)&:=\{\theta(W(D;C))\in A
\,|\,C\in\operatorname{Col}_X(D)_Y\}, \\
\mathcal{H}(D;\rho)&:=\{[W(D;C)]\in H_2(X)_Y
\,|\,C\in\operatorname{Col}_X(D;\rho)_Y\}, \\
\Phi_\theta(D;\rho)&:=\{\theta(W(D;C))\in A
\,|\,C\in\operatorname{Col}_X(D;\rho)_Y\}
\end{align*}
as multisets.

\begin{lemma} \label{lem:for_conjugation}
Let $D$ be a diagram of an oriented spatial trivalent graph $K$.
For $\rho,\rho'\in\operatorname{Hom}(G_K,G)$ such that $\rho$ and
$\rho'$ are conjugate, we have
$\mathcal{H}(D;\rho)=\mathcal{H}(D;\rho')$ and
$\Phi_\theta(D;\rho)=\Phi_\theta(D;\rho')$.
\end{lemma}

\begin{proof}
Let $g_0$ be an element of $G$ such that $\rho'(x)=g_0^{-1}\rho(x)g_0$
for any $x\in G_K$.
Fix $x_0\in X$.
We set $q_0:=(x_0,g_0)$.
Let $f:\operatorname{Col}_X(D;\rho)_Y\to\operatorname{Col}_X(D;\rho')_Y$
be the bijection defined by $f(C)(x)=C(x)*q_0$
(see Figure~\ref{fig:coloring_shift1}).

\begin{figure}
\mbox{}\hfill
\begin{minipage}{105pt}
\begin{picture}(105,80)
\put(10,35){\framebox(20,20){$y$}}
\put(60,35){\framebox(35,20){$y*q$}}
\put(45,15){\line(0,1){60}}
\put(46,25){\makebox(0,0){$\rightarrow$}}
\put(45,5){\makebox(0,0){$q$}}
\end{picture}
\end{minipage}
\hspace{5mm}$\stackrel{f}{\mapsto}$\hspace{5mm}
\begin{minipage}{135pt}
\begin{picture}(135,80)
\put(10,35){\framebox(35,20){$y*q_0$}}
\put(75,35){\framebox(50,20){$(y*q)*q_0$}}
\put(60,15){\line(0,1){60}}
\put(61,25){\makebox(0,0){$\rightarrow$}}
\put(60,5){\makebox(0,0){$q*q_0$}}
\end{picture}
\end{minipage}
\hfill\mbox{}
\caption{}
\label{fig:coloring_shift1}
\end{figure}

We prove $[W(D;C)]=[W(D;f(C))]\in H_2(X)_Y$ for
$C\in\operatorname{Col}_X(D;\rho)_Y$.
We assume that spatial trivalent graphs are drawn in
$\mathbb{R}^2(\subset S^2)$.
Let $D'$ be a diagram obtained from $D$ by putting an oriented loop
$\gamma$ in the outermost region $R_\infty$ so that the loop bounds a
disk, where the loop is oriented counterclockwise
(see Figure~\ref{fig:coloring_shift2}).
Let $C'$ be the $X_Y$-coloring of $D'$ defined by $C'(\gamma)=q_0$ and
$C'=C$ on $\mathcal{A}(D,D')\cup\mathcal{R}(D,D')$.
Then we note that $C'(R'_\infty)=C(R_\infty)*q_0$ for the region
$R'_\infty$ surrounded by the loop $\gamma$ in $D'$.
We deform the diagram $D'$ by using R2, R3 and R5 moves so that the loop
passes over all arcs of $D$ exactly once.
Then we denote by $D''$ and $C''\in\operatorname{Col}_X(D'')_Y$ the
resulting diagram and the corresponding $X_Y$-coloring of $D''$,
respectively.
We obtain the $X_Y$-coloring $f(C)$ from $C''$ by removing the loop from
$D''$, which also implies that $f$ is well-defined.

\begin{figure}
\mbox{}\hfill
\begin{minipage}{70pt}
\begin{picture}(70,60)
\put(50,40){\circle{34}}
\put(33,40){\makebox(0,0){$\rightarrow$}}
\put(0,33){\framebox(14,14){$D$}}
\put(35,5){\makebox(0,0){$D'$}}
\end{picture}
\end{minipage}
\hfill
\begin{minipage}{40pt}
\begin{picture}(40,60)
\put(20,40){\circle{34}}
\put(3,40){\makebox(0,0){$\rightarrow$}}
\put(13,33){\framebox(14,14){$D$}}
\put(20,5){\makebox(0,0){$D''$}}
\end{picture}
\end{minipage}
\hfill\mbox{}
\caption{}
\label{fig:coloring_shift2}
\end{figure}

Since no crossings increase or decrease when we add or remove the loop
$\gamma$, we have
\[ [W(D;C)]=[W(D';C')]=[W(D'';C'')]=[W(D;f(C))], \]
where the second equality follows from
Lemma~\ref{lem:W(D;C)_Rmove}.
Then we have
$\mathcal{H}(D;\rho)=\mathcal{H}(D;\rho')$ and
$\Phi_\theta(D;\rho)=\Phi_\theta(D;\rho')$.
\end{proof}

We denote by $\operatorname{Conj}(G_K,G)$ the set of conjugacy classes
of homomorphisms from $G_K$ to $G$.
By Lemma~\ref{lem:for_conjugation}, $\mathcal{H}(D;\rho)$ and
$\Phi_\theta(D;\rho)$ are well-defined for
$\rho\in\operatorname{Conj}(G_K,G)$.

\begin{lemma} \label{lem:for_orientation}
Let $D$ be a diagram of an oriented spatial trivalent graph $K$.
Let $E$ be a diagram obtained from $D$ by reversing the orientation of
an edge $e$.
For $\rho\in\operatorname{Hom}(G_K,G)$, we have
$\mathcal{H}(D)=\mathcal{H}(E)$, $\Phi_\theta(D)=\Phi_\theta(E)$,
$\mathcal{H}(D;\rho)=\mathcal{H}(E;\rho)$ and
$\Phi_\theta(D;\rho)=\Phi_\theta(E;\rho)$.
\end{lemma}

\begin{proof}
It is sufficient to show that
$\mathcal{H}(D;\rho)=\mathcal{H}(E;\rho)$.
We define a bijection
$f:\operatorname{Col}_X(D;\rho)_Y\to\operatorname{Col}_X(E;\rho)_Y$
by $f(C)(\alpha)=(p_X(C(\alpha)),p_G(C(\alpha))^{-1})$ if $\alpha$ is an
arc originates from the edge $e$, $f(C)(\alpha)=C(\alpha)$ otherwise.
We remark that $\rho_{f(C)}=\rho_C=\rho$.
The map $f$ is well-defined, since $z_1*(x,g)=z_2$ is equivalent to
$z_2*(x,g^{-1})=z_1$.
Then we have $w(\chi;C)=w(\chi;f(C))$ for every crossing $\chi$, since
we have
\begin{align*}
(y,(x_1,g_1),(x_2,g_2))
&=-(y*(x_1,g_1),(x_1,g_1^{-1}),(x_2,g_2)) \\
&=-(y*(x_2,g_2),(x_1,g_1)*(x_2,g_2),(x_2,g_2^{-1})) \\
&=((y*(x_1,g_1))*(x_2,g_2),(x_1,g_1^{-1})*(x_2,g_2),(x_2,g_2^{-1}))
\end{align*}
in $C_2(X)_Y$ (see Figure~\ref{fig:for_orientation}).
Then we have $\mathcal{H}(D;\rho)=\mathcal{H}(E;\rho)$.
\end{proof}

\begin{figure}
\begin{picture}(140,90)
\put(35,45){\line(1,0){25}}
\put(70,45){\line(1,0){25}}
\put(65,15){\line(0,1){60}}
\put(66,25){\makebox(0,0){$\rightarrow$}}
\put(45,46){\makebox(0,0){$\uparrow$}}
\put(17,45){\makebox(0,0){$(x_1,g_1)$}}
\put(65,10){\makebox(0,0){$(x_2,g_2)$}}
\put(40,20){\framebox(15,15){$y$}}
\end{picture}
\hfill
\begin{picture}(180,90)
\put(35,45){\line(1,0){25}}
\put(70,45){\line(1,0){25}}
\put(65,15){\line(0,1){60}}
\put(66,25){\makebox(0,0){$\rightarrow$}}
\put(45,44){\makebox(0,0){$\downarrow$}}
\put(13,45){\makebox(0,0){$(x_1,g_1^{-1})$}}
\put(65,10){\makebox(0,0){$(x_2,g_2)$}}
\put(0,55){\framebox(55,15){$y*(x_1,g_1)$}}
\end{picture}%
\\
\begin{picture}(140,90)
\put(35,45){\line(1,0){25}}
\put(70,45){\line(1,0){25}}
\put(65,15){\line(0,1){60}}
\put(64,25){\makebox(0,0){$\leftarrow$}}
\put(85,46){\makebox(0,0){$\uparrow$}}
\put(134,45){\makebox(0,0){$(x_1,g_1)*(x_2,g_2)$}}
\put(65,10){\makebox(0,0){$(x_2,g_2^{-1})$}}
\put(75,20){\framebox(55,15){$y*(x_2,g_2)$}}
\end{picture}
\hfill
\begin{picture}(180,90)
\put(35,45){\line(1,0){25}}
\put(70,45){\line(1,0){25}}
\put(65,15){\line(0,1){60}}
\put(64,25){\makebox(0,0){$\leftarrow$}}
\put(85,44){\makebox(0,0){$\downarrow$}}
\put(137,45){\makebox(0,0){$(x_1,g_1^{-1})*(x_2,g_2)$}}
\put(65,10){\makebox(0,0){$(x_2,g_2^{-1})$}}
\put(75,55){\framebox(100,15){$(y*(x_1,g_1))*(x_2,g_2)$}}
\end{picture}
\caption{}
\label{fig:for_orientation}
\end{figure}

By Lemma~\ref{lem:for_orientation}, 
$\mathcal{H}(D)$, $\Phi_\theta(D)$, $\mathcal{H}(D;\rho)$ and
$\Phi_\theta(D;\rho)$ are well-defined for a diagram $D$ of an
unoriented spatial trivalent graph, which is a diagram of a
handlebody-link.
For a diagram $D$ of a handlebody-link $H$, we define
\begin{align*}
\mathcal{H}^{\rm hom}(D)&:=
\{\mathcal{H}(D;\rho)\,|\,\rho\in\operatorname{Hom}(G_H,G)\}, \\
\Phi_\theta^{\rm hom}(D)&:=
\{\Phi_\theta(D;\rho)\,|\,\rho\in\operatorname{Hom}(G_H,G)\}, \\
\mathcal{H}^{\rm conj}(D)&:=
\{\mathcal{H}(D;\rho)\,|\,\rho\in\operatorname{Conj}(G_H,G)\}, \\
\Phi_\theta^{\rm conj}(D)&:=
\{\Phi_\theta(D;\rho)\,|\,\rho\in\operatorname{Conj}(G_H,G)\}
\end{align*}
as ``multisets of multisets''.
We remark that, for $X_Y$-colorings $C$ and $C_{D,E}$ in
Lemma~\ref{lem:W(D;C)_Rmove}, we have $\rho_C=\rho_{C_{D,E}}$.
Then, by Lemmas~\ref{lem:W(D;C)_2-cycle}--\ref{lem:for_orientation}, we
have the following theorem.

\begin{theorem} \label{thm:invariants}
Let $X$ be a $G$-family of quandles, $Y$ an $X$-set.
Let $\theta$ be a $2$-cocycle of $C^*(X;A)_Y$.
Let $H$ be a handlebody-link represented by a diagram $D$.
Then the following are invariants of a handlebody-link $H$.
\begin{align*}
&\mathcal{H}(D),&&\Phi_\theta(D),&&
\mathcal{H}^{\rm hom}(D),&&\Phi_\theta^{\rm hom}(D),&&
\mathcal{H}^{\rm conj}(D),&&\Phi_\theta^{\rm conj}(D).
\end{align*}
\end{theorem}

We denote the invariants of $H$ given in Theorem~\ref{thm:invariants} by
\begin{align*}
&\mathcal{H}(H),&&\Phi_\theta(H),&&
\mathcal{H}^{\rm hom}(H),&&\Phi_\theta^{\rm hom}(H),&&
\mathcal{H}^{\rm conj}(H),&&\Phi_\theta^{\rm conj}(H),
\end{align*}
respectively.

Let $\{y\}$ be a trivial $X$-set.
For the trivial $2$-cocycle $0$ of $C^*(X;A)_{\{y\}}$, we have
\begin{align*}
\Phi_0(H)
&=\{0\,|\,C\in\operatorname{Col}_X(D)_{\{y\}}\}, \\
\Phi_0^{\rm hom}(H)
&=\{\{0\,|\,C\in\operatorname{Col}_X(D;\rho)_{\{y\}}\}\,|\,
\rho\in\operatorname{Hom}(G_H,G)\}, \\
\Phi_0^{\rm conj}(H)
&=\{\{0\,|\,C\in\operatorname{Col}_X(D;\rho)_{\{y\}}\}\,|\,
\rho\in\operatorname{Conj}(G_H,G)\}.
\end{align*}
Thus
\begin{align*}
\#\operatorname{Col}_X(H)
&:=\#\operatorname{Col}_X(D)_{\{y\}}, \\
\#\operatorname{Col}_X^{\rm hom}(H)
&:=\{\#\operatorname{Col}_X(D;\rho)_{\{y\}}\,|\,
\rho\in\operatorname{Hom}(G_H,G)\}, \\
\#\operatorname{Col}_X^{\rm conj}(H)
&:=\{\#\operatorname{Col}_X(D;\rho)_{\{y\}}\,|\,
\rho\in\operatorname{Conj}(G_H,G)\}
\end{align*}
are invariants of a handlebody-link $H$ represented by a diagram $D$,
where $\#S$ denotes the cardinality of a multiset $S$.
We remark that these invariants do not depend on the choice of the
singleton set $\{y\}$.

We denote by $H^*$ the mirror image of a handlebody-link $H$.
Then we have the following theorem.

\begin{theorem}\label{theorem:mirror}
For a handlebody-link $H$, we have
\begin{align*}
&\mathcal{H}(H^*)=-\mathcal{H}(H), &
&\Phi_\theta(H^*)=-\Phi_\theta(H), \\
&\mathcal{H}^{\rm hom}(H^*)=-\mathcal{H}^{\rm hom}(H), &
&\Phi_\theta^{\rm hom}(H^*)=-\Phi_\theta^{\rm hom}(H), \\
&\mathcal{H}^{\rm conj}(H^*)=-\mathcal{H}^{\rm conj}(H), &
&\Phi_\theta^{\rm conj}(H^*)=-\Phi_\theta^{\rm conj}(H),
\end{align*}
where $-S=\{-a\,|\,a\in S\}$ for a multiset $S$.
\end{theorem}

\begin{proof}
Let $D$ be a diagram of a handlebody-link $H$.
We suppose that $D$ is depicted in an $xy$-plane $\mathbb{R}^2$.
Let $\varphi:\mathbb{R}^2\to\mathbb{R}^2$ be the involution defined by
$\varphi(x,y)=(-x,y)$.
Let $\widetilde{\varphi}:S^3\to S^3$ be the involution defined by
$\varphi(x,y,z)=(-x,y,z)$ and $\varphi(\infty)=\infty$, where we regard
the $3$-sphere $S^3$ as $\mathbb{R}^3\cup\{\infty\}$.
Then $\varphi(D)$ is a diagram of the handlebody-link
$H^*=\widetilde{\varphi}(H)$.
For $\rho\in\operatorname{Hom}(G_H,G)$ and
$C\in\operatorname{Col}_X(D;\rho)_Y$, we have
$\widetilde{\varphi}_*(\rho)\in\operatorname{Hom}(G_{H^*},G)$ and
$C\circ\varphi\in\operatorname{Col}_X(\varphi(D);\widetilde{\varphi}_*(\rho))_Y$,
where $\widetilde{\varphi}_*$ is the isomorphism induced by
$\widetilde{\varphi}$.
For each crossing $\chi$ of $D$,
$\epsilon(\chi)=-\epsilon(\varphi(\chi))$, and hence we have
$w(\varphi(\chi),C\circ\varphi)=-w(\chi,C)$.
Then $[W(\varphi(D);C\circ\varphi)]=-[W(D;C)]$, which implies the
equalities in this theorem.
\end{proof}

\section{Applications}
\label{sec:applications}

In this section, we calculate cocycle invariants defined in the previous
section for the handlebody-knots $0_1,\ldots,6_{16}$ in the table given
in~\cite{IshiiKishimotoMoriuchiSuzuki11}, by using a 2-cocycle given by
Nosaka \cite{Nosaka}.
This calculation enables us to distinguish some of handlebody-knots from
their mirror images, and a pair of handlebody-knots whose complements
have isomorphic fundamental groups.

Let $G=SL(2;\mathbb{Z}_3)$ and $X=(\mathbb{Z}_3)^2$.
Then $X$ is a $G$-family of quandles with the proper binary operation as
given in Proposition~\ref{prop:examples_G-family} (2).
Let $Y$ be the trivial $X$-set $\{y\}$.
We define a map $\theta:Y\times(X\times G)^2\rightarrow \mathbb{Z}_3$ by
\[ \theta(y,(x_1,g_1),(x_2,g_2))
:=\lambda(g_1)\det(x_1-x_2,x_2(1-g_2^{-1})), \]
where the abelianization $\lambda:G\rightarrow \mathbb{Z}_3$ is given by
\[ \lambda\begin{pmatrix}a&b\\c&d\end{pmatrix}=(a+d)(b-c)(1-bc). \]
By \cite{Nosaka}, the map $\theta$ is a 2-cocycle of
$C^*(X;\mathbb{Z}_3)_Y$.
Table~\ref{table-appl} lists the invariant $\Phi_{\theta}^{\rm conj}(H)$
for the handlebody-knots $0_1,\dots,6_{16}$.
We represent the multiplicity of elements of a multiset by using
subscripts.
For example, $\{\{0_2,1_3\}_1,\{0_3\}_2\}$ represents the multiset
$\{\{0,0,1,1,1\},\{0,0,0\},\{0,0,0\}\}$.

\begin{table}
\begin{center}
\begin{tabular}{|c||l|}
\hline
& $\Phi_{\theta}(H)$
\\ \hline
$0_1$ & $\{\{0_9\}_{76}\}$
\\ \hline
$4_1$ & $\{\{0_9\}_{83},\{0_{27}\}_{22},\{0_{81}\}_3\}$
\\ \hline
$5_1$ & $\{\{0_9\}_{76}\}$
\\ \hline
$5_2$ & $\{\{0_9\}_{95},\{0_{27}\}_6,\{0_{81}\}_1,\{0_9,1_{18}\}_4,\{0_{27},1_{54}\}_2\}$
\\ \hline
$5_3$ & $\{\{0_9\}_{102},\{0_{27}\}_4,\{0_{27},2_{54}\}_2\}$
\\ \hline
$5_4$ & $\{\{0_9\}_{74},\{0_{81}\}_2\}$
\\ \hline
$6_1$ & $\{\{0_9\}_{91},\{0_{27}\}_{16},\{0_{81}\}_1\}$
\\ \hline
$6_2$ & $\{\{0_9\}_{106},\{0_{45},1_{18},2_{18}\}_2\}$
\\ \hline
$6_3$ & $\{\{0_9\}_{74},\{0_{27}\}_2\}$
\\ \hline
$6_4$ & $\{\{0_9\}_{76}\}$
\\ \hline
$6_5$ & $\{\{0_9\}_{74},\{0_9,1_{18}\}_2\}$
\\ \hline
$6_6$ & $\{\{0_9\}_{72},\{0_{27}\}_4\}$
\\ \hline
$6_7$ & $\{\{0_9\}_{85},\{0_{27}\}_{16},\{0_{81}\}_3,\{0_{45},1_{18},2_{18}\}_4\}$
\\ \hline
$6_8$ & $\{\{0_9\}_{76}\}$
\\ \hline
$6_9$ & $\{\{0_9\}_{91},\{0_{27}\}_6,\{0_{81}\}_1,\{0_9,1_{18}\}_6,\{0_{27},1_{54}\}_2,\{0_{27},2_{54}\}_2\}$
\\ \hline
$6_{10}$ & $\{\{0_9\}_{76}\}$
\\ \hline
$6_{11}$ & $\{\{0_9\}_{70},\{0_9,1_{18}\}_6\}$
\\ \hline
$6_{12}$ & $\{\{0_9\}_{97},\{0_{81}\}_1,\{0_{9},1_{18}\}_8,\{0_{9},1_{36},2_{36}\}_2\}$
\\ \hline
$6_{13}$ & $\{\{0_9\}_{95},\{0_{27}\}_6,\{0_{81}\}_1,\{0_{9},2_{18}\}_4,\{0_{27},2_{54}\}_2\}$
\\ \hline
$6_{14}$ & $\{\{0_9\}_{119},\{0_{27}\}_6,\{0_{81}\}_{11},\{0_{9},1_{18}\}_{12},\{0_{27},1_{54}\}_{24}\}$
\\ \hline
$6_{15}$ & $\{\{0_9\}_{119},\{0_{27}\}_6,\{0_{81}\}_{11},\{0_{9},2_{18}\}_{12},\{0_{27},1_{54}\}_{24}\}$
\\ \hline
$6_{16}$ & $\{\{0_9\}_{44},\{0_{81}\}_{32}\}$
\\ \hline
\end{tabular}
\caption{}
\label{table-appl}
\end{center}
\end{table}

From Table~\ref{table-appl}, we see that our invariant can distinguish
the handlebody-knots $6_{14}$, $6_{15}$, whose complements have the
isomorphic fundamental groups.
Together with Theorem~\ref{theorem:mirror}, we also see that
handlebody-knots $5_2$, $5_3$, $6_5$, $6_9$, $6_{11}$, $6_{12}$,
$6_{13}$, $6_{14}$, $6_{15}$ are not equivalent to their mirror images.
In particular, the chiralities of $5_3$, $6_5$, $6_{11}$ and $6_{12}$
were not known.
Table~\ref{table-chiral} shows us known facts on the chirality of
handlebody-knots in \cite{IshiiKishimotoMoriuchiSuzuki11} so far.
In the column of ``chirality'', the symbols $\bigcirc$ and $\times$ mean
that the handlebody-knot is amphichiral and chiral, respectively, and
the symbol $?$ means that it is not known whether the handlebody-knot is
amphichiral or chiral.
The symbols $\checkmark$ in the right five columns mean that the
handlebody-knots can be proved chiral by using the method introduced in
the papers corresponding to the columns.
Here, M, II, LL, IKO and IIJO denote the papers \cite{Motto},
\cite{IshiiIwakiri12}, \cite{LeeLee}, \cite{IshiiKishimotoOzawa} and
this paper, respectively.

\begin{table}
\begin{center}
\begin{tabular}{|c||c||c|c|c|c|c|}
\hline
 & {chirality} & M & II & LL & IKO & IIJO
\\ \hline
$0_1$ & $\bigcirc$ & & & & &
\\ \hline
$4_1$ & $\bigcirc$ & & & & &
\\ \hline
$5_1$ & $\times$ & & & $\checkmark$ & &
\\ \hline
$5_2$ & $\times$ & & $\checkmark$ & $\checkmark$ & & $\checkmark$
\\ \hline
$5_3$ & $\times$ & & & & & $\checkmark$
\\ \hline
$5_4$ & $\times$ & & & & $\checkmark$ &
\\ \hline
$6_1$ & $\times$ & $\checkmark$ & & & &
\\ \hline
$6_2$ & $?$ & & & & &
\\ \hline
$6_3$ & $?$ & & & & &
\\ \hline
$6_4$ & $\times$ & & & $\checkmark$ & &
\\ \hline
$6_5$ & $\times$ & & & & & $\checkmark$
\\ \hline
$6_6$ & $\bigcirc$ & & & & &
\\ \hline
$6_7$ & $\bigcirc$ & & & & &
\\ \hline
$6_8$ & $?$ & & & & &
\\ \hline
$6_9$ & $\times$ & & $\checkmark$ & & & $\checkmark$
\\ \hline
$6_{10}$ & $?$ & & & & &
\\ \hline
$6_{11}$ & $\times$ & & & & & $\checkmark$
\\ \hline
$6_{12}$ & $\times$ & & & & & $\checkmark$
\\ \hline
$6_{13}$ & $\times$ & & $\checkmark$ & $\checkmark$ & & $\checkmark$
\\ \hline
$6_{14}$ & $\times$ & & & & $\checkmark$ & $\checkmark$
\\ \hline
$6_{15}$ & $\times$ & & & & $\checkmark$ & $\checkmark$
\\ \hline
$6_{16}$ & $\bigcirc$ & & & & &
\\ \hline
\end{tabular}
\caption{}
\label{table-chiral}
\end{center}
\end{table}

\section{A generalization}
\label{sec:generalization}

In this section, we show that our invariant is a generalization of the
invariant $\Phi_\theta^{\rm I}(H)$ defined by the first and second
authors in~\cite{IshiiIwakiri12}.
We refer the reader to~\cite{IshiiIwakiri12} for the details of the
invariant $\Phi_\theta^{\rm I}(H)$.
We recall the definition of the chain complex for the invariant
$\Phi_\theta^{\rm I}(H)$.

Let $X$ be a $\mathbb{Z}_m$-family of quandles, $Y$ an $X$-set.
Let $B^{\rm I}_n(X)_Y$ be the free abelian group generated by the
elements of $Y\times X^n$ if $n\geq0$, and let $B^{\rm I}_n(X)_Y=0$
otherwise.
We put
\[ ((y,x_1,\ldots,x_i)*^jx,x_{i+1},\ldots,x_n)
:=(y*^jx,x_1*^jx,\ldots,x_i*^jx,x_{i+1},\ldots,x_n) \]
for $y\in Y$, $x,x_1,\ldots,x_n\in X$ and $j\in\mathbb{Z}_m$.
We define a boundary homomorphism
$\partial_n:B^{\rm I}_n(X)_Y\to B^{\rm I}_{n-1}(X)_Y$ by
\begin{align*}
\partial_n(y,x_1,\ldots,x_n)
=&\sum_{i=1}^{n}(-1)^i(y,x_1,\ldots,x_{i-1},x_{i+1},\ldots,x_n) \\
&-\sum_{i=1}^{n}(-1)^i((y,x_1,\ldots,x_{i-1})*^1x_i,x_{i+1},\ldots,x_n)
\end{align*}
for $n>0$, and $\partial_n=0$ otherwise.
Then $B^{\rm I}_*(X)_Y=(B^{\rm I}_n(X)_Y,\partial_n)$ is a chain complex.
Let $D^{\rm I}_n(X)_Y$ be the subgroup of $B^{\rm I}_n(X)_Y$ generated
by the elements of
\[ \bigcup_{i=1}^{n-1}\left\{
(y,x_1,\ldots,x_{i-1},x,x,x_{i+2},\ldots,x_n)
\,\left|\,y\in Y,\,x,x_1,\ldots,x_n\in X
\right.\right\} \]
and
\[ \bigcup_{i=1}^n\left\{\left.
\sum_{j=0}^{m-1}((y,x_1,\ldots,x_{i-1})*^jx_i,x_i,\ldots,x_n)
\,\right|\,y\in Y,\,x_1,\ldots,x_n\in X
\right\}. \]
Then $D^{\rm I}_*(X)_Y=(D^{\rm I}_n(X)_Y,\partial_n)$ is a chain
complex.

We put $C^{\rm I}_n(X)_Y=B^{\rm I}_n(X)_Y/D^{\rm I}_n(X)_Y$.
Then $C^{\rm I}_*(X)_Y=(C^{\rm I}_n(X)_Y,\partial_n)$ is a chain
complex.
For an abelian group $A$, we define the cochain complex
$C_{\rm I}^*(X;A)_Y=\operatorname{Hom}(C^{\rm I}_*(X)_Y,A)$.
We denote by $H^{\rm I}_n(X)_Y$ the $n$th homology group of
$C^{\rm I}_*(X)_Y$.

\begin{proposition} \label{prop:HcongH}
For $n\in\mathbb{Z}$, we have
\[ H^{\rm I}_n(X)_Y\cong H_n(X)_Y. \]
\end{proposition}

\begin{proof}
The homomorphism $f_n:C^{\rm I}_n(X)_Y\to C_n(X)_Y$ defined by
\[ f_n((y,x_1,\ldots,x_n))=(y,(x_1,1),\ldots,(x_n,1)) \]
is an isomorphism, since the homomorphism
$g_n:C_n(X)_Y\to C^{\rm I}_n(X)_Y$ defined by
\begin{align*}
&g_n(y,(x_1,s_1),\ldots,(x_n,s_n)) \\
&=\sum_{i_1=0}^{s_1-1}\sum_{i_2=0}^{s_2-1}\cdots\sum_{i_n=0}^{s_n-1}
(\cdots((y*^{i_1}x_1,x_1)*^{i_2}x_2,x_2)\cdots *^{i_n}x_n,x_n)
\end{align*}
is the inverse map of $f_n$.
It is easy to see that $f=\{f_n\}$ is a chain map from
$C^{\rm I}_*(X)_Y$ to $C_*(X)_Y$.
Therefore $H^{\rm I}_n(X)_Y\cong H_n(X)_Y$.
\end{proof}

For a $2$-cocycle $\theta$ of $C_{\rm I}^*(X;A)_Y$, the composition
$\theta\circ g_2$ is a $2$-cocycle of $C^*(X;A)_Y$, and we have
\[ \Phi_\theta^{\rm I}(H)=\Phi_{\theta\circ g_2}^{\rm hom}(H), \]
where $g_2$ is the map defined in Proposition~\ref{prop:HcongH}.
Then our invariant is a generalization of the invariant introduced
in~\cite{IshiiIwakiri12}.

\section*{Acknowledgments}

The authors would like to thank Takefumi Nosaka for valuable discussions
about his cocycles.


\begin{thebibliography}{000}

\bibitem{CarterJelsovskyKamadaLangfordSaito03}
 J.~S.~Carter, D.~Jelsovsky, S.~Kamada, L.~Langford and M.~Saito,
 \textit{Quandle cohomology and state-sum invariants of knotted curves and surfaces},
 Trans. Amer. Math. Soc. \textbf{355} (2003) 3947--3989.

\bibitem{CarterJelsovskyKamadaSaito01}
 J.~S.~Carter, D.~Jelsovsky, S.~Kamada and M.~Saito,
 \textit{Quandle homology groups, their Betti numbers, and virtual knots},
 J. Pure Appl. Algebra \textbf{157} (2001) 135--155.

\bibitem{CarterKamadaSaito01} 
 J.~S.~Carter, S.~Kamada, and M.~Saito, 
 \textit{Geometric interpretations of quandle homology}, 
 J. Knot Theory Ramifications \textbf{10} (2001), 345--386.

\bibitem{FennRourkeSanderson95}
 R.~Fenn, C.~Rourke and B.~Sanderson,
 \textit{Trunks and classifying spaces},
 Appl. Categ. Structures \textbf{3} (1995), 321--356.

\bibitem{FennRourkeSanderson07}
 R.~Fenn, C.~Rourke and B.~Sanderson,
 \textit{The rack space},
 Trans. Amer. Math. Soc. \textbf{359} (2007), 701--740.

\bibitem{Ishii08}
 A.~Ishii,
 \textit{Moves and invariants for knotted handlebodies},
 Algebr. Geom. Topol. \textbf{8} (2008), 1403--1418.

\bibitem{IshiiIwakiri12}
 A.~Ishii and M.~Iwakiri,
 \textit{Quandle cocycle invariants for spatial graphs and knotted handlebodies},
 Canad. J.~Math. 64 (2012), 102--122.

\bibitem{IshiiKishimotoMoriuchiSuzuki11}
 A.~Ishii, K.~Kishimoto, H.~Moriuchi and M.~Suzuki,
 \textit{A table of genus two handlebody-knots up to six crossings},
 to appear in J. Knot Theory Ramifications

\bibitem{IshiiKishimotoOzawa}
 A.~Ishii, K.~Kishimoto and M.~Ozawa,
 \textit{Knotted handle decomposing spheres for handlebody-knots},
 preprint.

\bibitem{JangOshiro}
 Y.~Jang and K.~Oshiro, 
 \textit{Symmetric quandle colorings for spatial graphs and handlebody-links}, 
 J. Knot Theory Ramifications.

\bibitem{Joyce82}
 D.~Joyce,
 \textit{A classifying invariant of knots, the knot quandle},
 J. Pure Appl. Alg. \textbf{23} (1982) 37--65.

\bibitem{Kamada07} 
 S. Kamada,
 \textit{Quandles with good involutions, their homologies and knot invariants},
 in: Intelligence of Low Dimensional Topology 2006, Eds. J. S. Carter
 et. al., pp. 101--108, World Scientific Publishing Co., 2007.

\bibitem{KamadaOshiro} 
S. Kamada and K. Oshiro, 
{\it Homology groups of symmetric quandles and cocycle invariants of links and surface-links}, 
Trans. Amer. Math. Soc. {\bf 362} (2010) 5501--5527.  

\bibitem{LeeLee}
 J.~H.~Lee and S.~Lee,
 \textit{Inequivalent handlebody-knots with homeomorphic complements},
preprint.

\bibitem{Matveev82}
 S.~V.~Matveev,
 \textit{Distributive groupoids in knot theory},
 Mat. Sb. (N.S.) \textbf{119(161)} (1982) 78--88.
  
\bibitem{Motto}
 M.~Motto,
 \textit{Inequivalent genus 2 handlebodies in $S^3$ with homeomorphic complement},
Topology Appl. {\bf 36} (1970), 283--290.

\bibitem{Nosaka}
 T.~Nosaka,
 \textit{Quandle cocycles from invariant theory},
 preprint.

\bibitem{RourkeSanderson}
 C.~Rourke and B.~Sanderson, 
 \textit{There are two 2-twist-spun trefoils}.
 http://citeseerx.ist.psu.edu/viewdoc/summary?doi=10.1.1.65.3250

\end{thebibliography}
\end{document}